\documentclass[a4paper,10pt]{article}
\usepackage{geometry}
\usepackage[applemac]{inputenc}
\usepackage[T1]{fontenc}
\usepackage{amsmath,amsfonts, amssymb,enumerate,amsthm,amscd,graphicx,color,url}
\usepackage[english,frenchb]{babel}
\usepackage[all]{xy}
\newcommand{\R}{\mathbb{R}}

\newcommand{\N}{\mathbb{N}}
\newcommand{\Z}{\mathbb{Z}}

\newcommand{\F}{\mathbb{F}}
\newcommand{\Or}{\mathcal{O}}
\newcommand{\Mr}{\mathfrak{m}}
\newcommand{\Ar}{\mathfrak{a}}

\newtheorem{theo}{Théorème}

{\theoremstyle{plain}
\newtheorem{thm}{Théorème}[section]

\newtheorem{prop}[thm]{Proposition}
\newtheorem{cor}[thm]{Corollaire}}
{\theoremstyle{definition}
\newtheorem{defi}{Définition}[section]

\newtheorem{rmq}{Remarque}[section]}
\newcommand{\gal}{\textrm{Gal}}
\let\urlorig\url
\renewcommand{\url}[1]{
   \begin{otherlanguage}{english}\urlorig{#1}\end{otherlanguage}}
\begin{document}
\title{Optimisation du théorème d'\textsc{Ax-Sen-Tate} et application à un calcul de cohomologie galoisienne $p$-adique.}
\author{\scshape Jérémy Le Borgne\footnote{IRMAR, Université de Rennes 1, Campus de Beaulieu, 35042 Rennes Cedex, France. jeremy.le-borgne@univ-rennes1.fr}}
\date{}
\maketitle
\tableofcontents
\section*{Introduction}
Soit $p$ un nombre premier, $k$ un corps parfait de caractéristique $p$, et $F = \textrm{Frac }\textrm{W}(k)$. Soit $K$ une extension finie totalement ramifiée de $F$. On note $\Or_K$ l'anneau des entiers de $K$, $\Mr_K$ son unique idéal maximal, et $\pi_K$ (ou $\pi$ s'il n'y a pas de confusion possible) une uniformisante de $K$. L'indice de ramification de $K$ sur $F$ est noté $e$ (c'est le degré de $K$ sur $F$). Enfin, on note $\bar{K}$ une clôture algébrique de $K$, et $C$ le complété de $\bar{K}$, auquel on étend la valuation de $F$ notée $v$, et normalisée par $v(p) = 1$.
Le théorème d'\textsc{Ax-Sen-Tate} dit que les points fixes de $C$ sous l'action de $G = \gal(\bar{K}/K)$ sont exactement les éléments de $K$. La démonstration d'\textsc{Ax} (voir \cite{ax}) de ce théorème s'appuie sur le résultat suivant :
\begin{theo}[Ax]
Soit $x \in C$ et $A \in \R$. On suppose que pour tout $\sigma \in G$, $v(\sigma x - x) \geq A$. Alors, il existe $y \in K$ tel que $v(x - y) \geq  A - \frac{p}{(p-1)^2}$.
\end{theo}
\textsc{Ax} pose sans y répondre la question de l'optimalité de la constante $\frac{p}{(p-1)^2}$ intervenant dans le théorème précédent, en précisant qu'une borne inférieure pour cette constante optimale est effectivement $\frac{1}{p-1}$. Pour traiter cette question nous introduisons ici la tour d'extensions de $K$ par les racines $p^m$-ièmes de l'uniformisante $\pi$ de $K$ : $\pi_0 = \pi$, $\pi_{m+1}^p = \pi_m$ et $K_m = K(\pi_m)$ pour tout $m\geq 0$, et $K_\infty = \cup_{m\geq 0} K_m$. La première partie est consacrée à l'étude de l'extension $K_\infty/K$. Dans sa démonstration du théorème d'\textsc{Ax-Sen-Tate} (voir \cite{tat}), {\scshape Tate} présente des calculs de la cohomologie galoisienne à coefficients dans l'extension cyclotomique de $K$. Dans cet article, nous démontrons des résultats du même type lorsque $K_\infty$ est une extension infinie, arithmétiquement profinie (APF) de $K$ (intervenant dans les travaux de  {\scshape Fontaine}  et {\scshape Wintenberger} sur la théorie du corps des normes, voir \cite{win}).  Ces résultats s'appliquent à l'extension $K_\infty$, et une étude \emph{ad hoc} de l'extension $K_\infty /K$, qui constitue la partie essentiellement originale de cet article, nous permettra de démontrer le :
\begin{theo}[Théorème \ref{ax1}]
Soit $x \in C$ et $A \in \R$. Si on suppose que pour tout $\sigma \in G$, $v(\sigma x - x) \geq A$, alors, pour tout $m \in \N$, il existe $y_m \in K_m$ tel que $v(x - y_m) \geq  A - \frac{1}{p^m(p-1)}$. Réciproquement, si pour tout $m \in \N$ il existe $y_m \in K_m$ tel que $v(x - y_m) \geq A - \frac{1}{p^m(p-1)}$, alors pour tout $\sigma \in G$, $v(\sigma x -x) \geq A$.
\end{theo}
Le cas particulier où $m=0$ implique que la constante optimale dans le théorème d'\textsc{Ax} est $\frac{1}{p-1}$.  Le cas $m=1$ est utilisé par \textsc{Caruso} (voir \cite{car}, théorème 3.5.4) pour prouver une formule de réciprocité explicite entre un $(\varphi,N)$-module filtré de torsion et la $\F_p$-représentation de $G$ associée ; la version d'\textsc{Ax} du théorème, et même le cas $m=0$ de notre théorème \ref{ax1}, s'avèrent insuffisants pour l'utilisation faite par \textsc{Caruso} dans le cas général.\\
La caractérisation donnée par le théorème nous permet en outre de décrire la structure de $H^1(G, \Or_{\bar{K}})$. La deuxième partie est dédiée à cette étude. Nous redémontrons notamment un résultat dû à \textsc{Sen} (voir \cite{sen}, théorème 3) qui dit que si l'entier $n$ vérifie $n\geq \frac{e}{p-1}$, alors $H^1(G, \Or_{\bar{K}})$ est tué par $\pi_K^n$. Dans le cas où $K=F$, on montre que $H^1(G, \Or_{\bar{K}})$ est isomorphe au sous-espace de $k^\N$ formée des suites vérifiant une relation de récurrence linéaire tordue par le Frobenius, introduites sous le nom de suites \emph{twist-récurrentes} par \textsc{Kedlaya}  dans le but de donner une description de $\bar{K}$. On donne finalement quelques indications pour obtenir une description analogue dans le cas ramifié, qui pourrait être le point de départ à un analogue en torsion de la théorie de Sen.

\section{Optimisation du théorème d'\textsc{Ax}.}
Nous démontrons dans cette partie le théorème \ref{ax1} de l'introduction. Nous introduisons l'extension de $K$ engendrée par les racines d'ordre une puissance de $p$ de $\pi$, que nous étudions à l'aide de la théorie des extensions APF. Nous étudions en détail les propriétés de $K_\infty$, et nous en déduisons une caractérisation des éléments de $K_\infty$ vérifiant une condition du type \og Pour tout $\sigma \in G$, $v(\sigma x - x) \geq A$ \fg.
\subsection{Extensions APF.}
Soit $K_\infty$ une extension infinie, arithmétiquement profinie (APF) (cf. \cite{win}, §1) de $K$. On se propose dans cette partie de décrire la cohomologie de $\gal(\bar{K}/K_\infty)$ à coefficients dans $C$, en cherchant à généraliser les résultats de \cite{tat} concernant l'extension cyclotomique de $K$. On utilise la numérotation supérieure des groupes de ramification, comme défini dans \cite{ser1}, chap. IV. Pour $\mu \geq -1$, si $M$ est une extension finie de $K$, on pose
 $$ M^\mu = M\cap \bar{K}^{\gal(\bar{K}/K)^\mu}. $$
Si $e$ désigne l'indice de ramification absolu de $K$, on sait d'après \cite{col1}, proposition 3.22, que
$$ v(\mathfrak{d}_{M/K}) = \frac{1}{e}\int_{-1}^{+\infty} \left(1 - \frac{1}{[M:M^\mu]}\right)d\mu. $$
\begin{prop}
Soit $L_\infty$ une extension galoisienne finie de $K_\infty$. Alors $Tr_{L_\infty/K_\infty}(\Or_{L_\infty}) \supset \Mr_{K_\infty}$.
\end{prop}
\begin{proof} Suivant \cite{win} §1, nous définissons comme dans \emph{loc. cit.} la suite $\mu_n$ comme la suite strictement croissante des $\mu \in \R_+$ tels que pour tout $\varepsilon > 0$, 
\begin{equation} \label{mun}
\gal(\bar{K}/K)^\mu\gal(\bar{K}/K_\infty) \neq \gal(\bar{K}/K)^{\mu+\varepsilon}\gal(\bar{K}/K_\infty),
\end{equation}
et nous notons $(K_n)_{n\in \N}$ la tour des extensions élémentaires de $K_\infty$ définie comme dans \cite{win} : $K_n$ est le sous-corps de $K_\infty$ fixé par $\gal(\bar{K}/K)^{\mu_n}\gal(\bar{K}/K_\infty)$.
Quitte à remplacer $K$ par l'un des $K_n$, on sait d'après \cite{ser1}, Chap V, §4, Lemme 6, qu'il existe une extension $L$ de $K$ linéairement disjointe de $K_\infty$ telle que $L_\infty = LK_\infty$. On note pour tout $n \in \N$, $L_n = LK_n$. La configuration des extensions considérées est la suivante :
$$
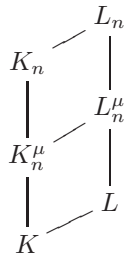
\xymatrix @R=1mm @C=1mm {
   &&L_n \\
    K_n \ar@{-}[rru]\\
    &&L_n^\mu \ar@{-}[uu] \\
    K_n^\mu \ar@{-}[rru] \ar@{-}[uu]\\
    && L \ar@{-}[uu]  \\
    K\ar@{-}[rru] \ar@{-}[uu] 
  }$$
Suivons pas à pas la méthode de \textsc{Colmez} (\cite{col2}, 1.4.) : pour tout $\mu \geq -1$, $[K_n:K_n^\mu] = [K_nL_n^\mu:L_n^\mu]$ (car $K_n^\mu = K_n\cap L_n^\mu$). On a bien sûr $ v(\mathfrak{d}_{L_n/K_n}) = v(\mathfrak{d}_{L_n/K}) - (\mathfrak{d}_{K_n/K})$. Ainsi, en appliquant la formule pour le calcul de la valuation de la différente, il vient :
\begin{eqnarray*}v(\mathfrak{d}_{L_n/K_n}) &= & \frac{1}{e}\int_{-1}^{+\infty} \left(\frac{1}{[K_n:K_n^\mu]} - \frac{1}{[L_n:L_n^\mu]} \right)d\mu\\
&= & \frac{1}{e[K_n:K]}\int_{-1}^{+\infty} [K_n^\mu:K]\left(1 - \frac{1}{[L_n:K_nL_n^\mu]} \right)d\mu.
\end{eqnarray*}
Soit $n_0$ un entier tel que $L^{n_0} = L$.
Si $\mu \geq n\geq n_0$, alors $L^\mu = L \subset L_n^\mu$. Ainsi, $L_n = K_nL \subset K_nL_n^\mu$, c'est à dire $[L_n : K_nL_n^\mu] = 1$. Par conséquent,
\begin{eqnarray*}
e[K_n:K]v(\mathfrak{d}_{L_n/K_n}) &=& \int_{-1}^{n_0} [K_n^\mu:K]\left(1 - \frac{1}{[L_n:K_nL_n^\mu]} \right)d\mu\\
&\leq& \int_{-1}^{n_0} [K_n^\mu : K]d\mu.
\end{eqnarray*}
Pour $\mu \leq \mu_n$, et pour $m\leq n$, $K_n^\mu = K_m^\mu$. Comme $K_\infty/K$ est APF, $\mu_n$ tend vers $+ \infty$, et on a pour $n$ assez grand :
$$\int_{-1}^{n_0} [K_n^\mu:K]d\mu = \int_{-1}^{n_0} [K_{n_0}^\mu:K]d\mu,$$
qui est une constante. Par conséquent, $v(\mathfrak{d}_{L_n/K_n}) = O\left(\frac{1}{[K_n:K]}\right)$. On a alors (\cite{ser1}, Chap V, §3, Lemme 4) :
$$ v(Tr_{L_n/K_n}(\Mr_{L_n})) = O([K_n:K]^{-1}), $$
et donc un élément de $\Mr_{K_\infty}$ se trouve dans $Tr_{L_n/K_n}(\Mr_{L_n})$ pour $n$ assez grand, et par conséquent dans $Tr_{L_\infty/K_\infty}(\Or_{L_\infty})$.
\end{proof}
Il résulte immédiatement de cette proposition qu'étant donné $\varepsilon > 0$, il existe $y_{\varepsilon} \in \Or_{L_\infty}$ tel que $v(Tr_{L_\infty/K_\infty}(y_\varepsilon)) < \varepsilon$.
\begin{cor}\label{zeps}
Soit $L_\infty$ une extension galoisienne finie de $K_\infty$ de groupe de Galois $G$, et soient $x \in L_\infty$ et $\varepsilon >0$. Alors il existe $y \in L_\infty$ tel que
$$ v(x - Tr_{L_\infty/K_\infty}(y)) \geq  \min_{\sigma \in G}v(\sigma x - x) -\varepsilon \textrm{ et } v(y) \geq   v(x)-\varepsilon.$$
\end{cor}
\begin{proof} Soit $y_\varepsilon \in \Or_{L_\infty}$ tel que $v(Tr_{L_\infty/K_\infty}(y_\varepsilon)) < \varepsilon$, et soit $z = Tr_{L_\infty/K_\infty}(y_\varepsilon)$.\\
On pose $y = \frac{1}{z} xy_\varepsilon$. Alors $v(y) \geq v(x) - \varepsilon$ et
$$Tr_{L_\infty/K_\infty}(y) = \frac{1}{z} \sum_{\sigma \in G} \sigma(y_\varepsilon)\sigma(x) = \frac{1}{z}\sum_{\sigma \in G} \sigma(y_0) (\sigma x -x + x) = x + \frac{1}{z}\sum_{\sigma \in G} \sigma(y_0) (\sigma x -x).$$ 
Ainsi, $ v(Tr_{L_\infty/K_\infty}(y) - x) \geq \min_{\sigma \in G} v(\sigma x-x) - \varepsilon.$
\end{proof}
Donnons au passage une proposition qui ne nous servira pas par la suite, mais qui découle directement du corollaire \ref{zeps} en reprenant les arguments de  \cite{tat}, dont elle généralise la proposition 10 aux extensions APF.
\begin{prop}
Soit $K_\infty/K$ une extension APF infinie, on note $\mathcal{H} = \gal(\bar{K}/K_\infty)$, et $\widehat{K_\infty}$ la fermeture de $K_\infty$ dans $C$. Alors :
$$H^0(\mathcal{H},C) = \widehat{K_\infty} \textrm{ et } H^r(\mathcal{H},C) = 0 \textrm{ pour } r\geq1. $$ 
\end{prop}
\subsection{Etude de l'extension $K_\infty/K$.}
On peut maintenant appliquer le résultat précédent à une extension APF bien choisie. On définit $\pi_0 =\pi$, pour tout $n\in \N$, $\pi_{n+1}$ une racine $p$-ième de $\pi_n$, $K_n = K(\pi_n)$, et $K_\infty = \bigcup_{n\in\N} K_n$. L'extension $K_\infty$, étudiée par \textsc{Breuil} dans \cite{bre} est APF. Nous redonnons ici une démonstration plus élémentaire (dans le sens ou elle n'a pas recours aux groupes de Lie $p$-adiques) de ce résultat, qui est le lemme 2.1.1 de \emph{loc. cit.}
\begin{prop} L'extension $K_\infty/K$ est APF.
\end{prop}
\begin{proof}
On sait que $v(\mathfrak{d}_{K_n/K}) = e(n+1) - \frac{e}{p^n}$. Un rapide calcul à partir de l'expression intégrale de $v(\mathfrak{d}_{K_n/K})$ et une récurrence immédiate montrent alors que pour tout $n \geq 1$, la famille $(\mu_n)_{n \in \N}$ étant définie comme en (\ref{mun}),
$$\mu_n =  ne-1 + \frac{pe}{p-1}.$$
La suite $\mu_n$ tend vers $+\infty$, il résulte de \cite{win}, 1.4.2. que $K_\infty/K$ est APF et que $(K_n)$ est la tour d'extensions élémentaires de $K_\infty$.
\end{proof}
On a le corollaire suivant :
\begin{cor} \label{coro}Soit $\varepsilon >0$, soit $x \in \bar{K}$ tel que pour tout $\sigma \in G$, $v(\sigma x - x) \geq A$. Il existe $y_\varepsilon \in K_\infty$ tel que $v(x - y_\varepsilon) \geq A - \varepsilon$.
\end{cor}
\begin{proof}
On note $M$ la clôture galoisienne de $K_\infty(x)$, alors $\gal(\bar{K}/M)$ agit trivialement sur $x$, et donc pour tout $\sigma \in \gal(M/K)$,  $v(\sigma x - x) \geq A$, et en particulier pour tout $\sigma \in \gal(M/K_\infty)$,  $v(\sigma x - x) \geq A$.
D'après le corollaire \ref{zeps}, il existe pour tout $\varepsilon >0$ un $z_\varepsilon \in K_\infty$ tel que
$$ v(x-z_\varepsilon) \geq A-\varepsilon \textrm{ et } v(z_\varepsilon) \geq v(x) - \varepsilon. \qedhere$$
\end{proof}

\begin{thm}\label{ax0}
Soit $x \in K_\infty$. Les deux assertions suivantes sont équivalentes :
\begin{enumerate}[(i)]
\item $\forall \sigma \in G = \gal(\bar{K}/K)$, $v(\sigma x - x)\geq A$
\item $\forall m \in \N$, $\exists y_m \in K_m$ tel que $v(x - y_m) \geq A - \frac{1}{p^m(p-1)}$.
\end{enumerate}
En particulier, si $x$ vérifie $(i)$, il existe $y \in K$ tel que $v(x - y) \geq A- \frac{1}{p-1}$.
\end{thm}
\begin{proof} 
Pour simplifier l'écriture, on notera $A_m = A- \frac{1}{p^m(p-1)}$.\\
$(i)\Rightarrow (ii)$. Soit $n \in \N$ tel que $x \in K_n$. L'élément $x$ s'écrit
$$ x = \sum_{i=0}^{p^n-1}a_i\pi_n^i,$$
avec les $a_i$ dans $K$. Pour $\sigma \in G$, $\sigma \pi_n$ est de la forme $\zeta \pi_n$, avec $\zeta$ racine $p^n$-ième de l'unité. De plus, il existe $\sigma \in G$ tel que $\zeta$ soit une racine \emph{primitive} $p^n$-ième de l'unité. Fixons un tel $\sigma$. On a :
$$ \sigma x - x = \sum_{i=1}^{p^n - 1} a_i\pi_n^i(\zeta^i - 1).$$

L'ordre de $\zeta^i$ en tant que racine de l'unité est $p^{n - v(i)}$. En effet, écrivant $i = p^{v(i)}d$ avec $d$ non divisible par $p$. Alors $\zeta^i = (\zeta^d)^{p^{v(i)}}$, $\zeta^d$ est une racine primitive $p^n$-ième de 1, donc $(\zeta^d)^{p^{v(i)}}$ est une racine primitive $p^{n - v(i)}$-ième de 1.

Par conséquent, pour $i \in \{1,\dots,p^n-1\}$, on a
$$ v(a_i\pi_n^i(\zeta^i - 1)) = v(a_i) + \frac{i}{ep^n} + \frac{1}{p^{n - v(i) - 1}(p-1)}.$$

Soient $i, j \in \{1,\dots,p^n-1\}$ tels que $v(a_i\pi_n^i(\zeta^i - 1)) = v(a_j\pi_n^j(\zeta^j - 1))$. Alors
$$ \frac{i - j}{ep^n} + \frac{p^{v(i)} - p^{v(j)}}{p^{n - 1}(p-1)} = v(a_j) - v(a_i) \in \frac{1}{e}\Z.$$
Les entiers $v(i)$ et $v(j)$ sont inférieurs ou égaux à $n-1$, on peut de plus supposer que $v(i) \leq v(j)$. On a alors $ (i-j)(p-1) - ep^{v(i) + 1} (1 - p^{v(j) - v(i)}) \in p^n(p-1)\Z$. Si $v(i) < v(j)$, alors 
$$v((i-j)(p-1)) = v(i)\textrm{, et }v(ep^{v(i) + 1} (1 - p^{v(j) - v(i)})) = v(e) + v(i) + 1.$$
Par conséquent,  $v((i-j)(p-1) - ep^{v(i) + 1} (1 - p^{v(j) - v(i)})) = v(i) < n-1.$

On a donc nécessairement $v(i) = v(j)$. Ainsi, $\frac{i - j}{p^n} \in \Z$, et comme $i$ et $j$ sont inférieurs strictement à $p^n$,  $i =j$. Finalement, 
$$ v(a_i\pi_n^i(\zeta^i - 1)) = v(a_j\pi_n^j(\zeta^j - 1)) \Leftrightarrow i = j.$$
En particulier,
$$ v(\sigma x - x) = \min_{1\leq i \leq p^n-1} \left(v(a_i) + \frac{i}{ep^n} + \frac{p^{v(i)}}{p^{n-1}(p-1)}\right). $$
Soit $0 \leq m \leq n-1$. D'après le calcul précédent, si $i \in \{1,\dots,p^n-1\}$ et $v(i) < n-m$, on a sous les hypothèses du théorème :
$$ v(a_i) + \frac{i}{ep^n} \geq A - \frac{p^{n-m-1}}{p^{n-1}(p-1)} \geq A_m.$$
D'autre part,
$$ y_m= \sum_{\substack{0 \leq j \leq p^n-1\\p^{n-m}\mid j}}a_j\pi_n^j = \sum_{j=0}^{p^m-1}a_{p^{n-m}j}\pi_m^j \in K_m. $$
Finalement, on a
$$ v(x-y_m) = \min_{\substack{1 \leq j \leq p^n-1\\p^{n-m}\nmid j}} v(a_j\pi_n^j)= \min_{\substack{1 \leq j \leq p^n-1\\v(j)<n-m}} \left(v(a_{j}) + \frac{j}{ep^n}\right) \geq A_m, $$
avec $y_m \in K_m$.\\
Pour $m \geq n$, il suffit de choisir $y_m = x$.\\

Réciproquement, $(ii) \Rightarrow (i)$ : On suppose que $x \in K_n$. On écrit $x = \sum_{i=0}^{p^n-1} a_i\pi_n^i$. On fixe $\sigma_0 \in \gal(\bar{K}/K)$ tel que $\sigma_0 \pi_n = \zeta \pi_n$ avec $\zeta$ racine primitive $p^n$-ième de l'unité. On pose pour tout $m < n$ :
$$ z_m =\sum_{j=0}^{p^m-1}a_{p^{n-m}j}\pi_m^j.$$
La démonstration comporte trois étapes : tout d'abord, on montre que $v(\sigma x - x) \geq v(\sigma_0 x - x)$ pour tout $\sigma \in \gal(\bar{K}/K)$. Ensuite, on vérifie que $v(x-z) \leq v(x-z_m)$ pour tout $z \in K_m$. Enfin, on montre qu'il existe $m < n$ tel que $v(x-z_m) \leq v(\sigma_0 x - x) - \frac{1}{p^m(p-1)}$, et on conclut. \\

Soit $\sigma \in G$. Il existe une racine $p^n$-ième de l'unité $\omega$ telle que $\sigma \pi_n = \omega \pi_n$. Notons $p^r$ l'ordre de $\omega$ en tant que racine de l'unité. On a alors $ \sigma x - x = \sum_{i=1}^{p^n-1}a_i\pi_n^i(\omega^i - 1)$.
Pour tout $i \in \{1,\dots,p-1\}$, $v(a_i\pi_n^i(\omega^i - 1)) = v(a_i) + \frac{i}{ep^n} + \frac{p^{v(i)}}{p^{r-1}(p-1)}$. Or $r\leq n$, donc
$$ \forall i \in \{1,\dots,p-1\}, ~ v(a_i\pi_n^i(\omega^i - 1)) \geq v(a_i) + \frac{i}{ep^n} + \frac{p^{v(i)}}{p^{n-1}(p-1)}. $$
Or on sait que $ v(\sigma_0 x - x) = \min_{1\leq i\leq p^n-1} \left(v(a_i)+ \frac{i}{ep^n} + \frac{p^{v(i)}}{p^{n-1}(p-1)}\right)$, et donc 
$$ v(\sigma x - x) \geq \min_{1\leq i\leq p^n-1}  \left(v(a_i\pi_n^i(\omega^i - 1))\right)\geq v(\sigma_0 x - x).$$

Soit $0 \leq m \leq n-1$, et soit $z\in K_m$. On va montrer que $v(z - z_m) \neq v(x-z_m)$. Tout d'abord, comme $z$ et $z_m$ sont dans $K_m$, $v(z - z_m) \in \frac{1}{ep^m}\Z$. D'autre part, $ v(x - z_m) = \min_{v(i)<n-m} \left(v(a_i) + \frac{i}{ep^n}\right)$. Ainsi,
$$ ep^mv(x - z_m) = \min_{v(i)<n-m} \left(ep^mv(a_i) + \frac{i}{p^{n-m}}\right). $$
Si $v(i) < n-m$, $\frac{i}{p^{n-m}} \notin \Z$. Comme $ev(a_i) \in \Z$ pour tout $i$, on en déduit que $ep^mv(x - z_m) \notin \Z$, et donc que $v(z - z_m) \neq v(x-z_m)$. Par conséquent, $ v(x-z) = \min(v(x-z_m),v(z-z_m)) \leq v(x-z_m)$.

On sait que $v(\sigma_0 x - x) = \min_{1 \leq i \leq p^n-1} \left(v(a_i)+ \frac{i}{ep^n} + \frac{p^{v(i)}}{p^{n-1}(p-1)}\right)$. Notons $i_0$ l'indice pour lequel ce minimum est atteint ; fixons $m\in\{0,\dots,n-1\}$ tel que $v(i_0) = n-1-m$. On a
$$ v(x-z_m) = \min_{v(i)<n-m}\left(v(a_i) + \frac{i}{ep^n}\right) \leq v(a_{i_0}) + \frac{i_0}{ep^n}.$$
Or, par définition de $i_0$, $ v(\sigma_0 x - x) = v(a_{i_0})+ \frac{i_0}{ep^n} + \frac{1}{p^m(p-1)}$, et donc $ v(x-z_m) \leq v(\sigma_0 x - x) - \frac{1}{p^m(p-1)}$.

Fixons-nous maintenant $\sigma \in G$. D'après les hypothèses du théorème, il existe un $y_m \in K_m$ tel que $ v(x-y_m) \geq A_m.$
On fixe un tel $y_m$, on a en particulier $ v(x-z_m) \geq A_m$.
Ainsi,
$$ v(\sigma x - x) \geq v(\sigma_0 x - x) \geq v(x-z_m) + \frac{1}{p^m(p-1)} \geq A. \qedhere$$
\end{proof}
\subsection{Optimisation du théorème d'\textsc{Ax}.}
Dans cette partie, on utilise les résultats de la partie précédente et de l'étude menée sur les extensions APF pour donner la constante optimale dans le théorème d'\textsc{Ax}. Remarquons d'emblée que la constante optimale est minorée par $\frac{1}{p-1}$. En effet, $v(\sigma \pi_1 - \pi_1) = v(\pi_1) + \frac{1}{p-1}$, et $\sup_{y \in K} v(\pi_1 - y) = v(\pi_1)$. On va montrer que $\frac{1}{p-1}$ est en fait la constante optimale.
\begin{thm}\label{ax1}
Soit $x \in C$. Les deux assertions suivantes sont équivalentes :
\begin{enumerate}[(i)]
\item $\forall \sigma \in G = \gal(\bar{K}/K)$, $v(\sigma x - x)\geq A$
\item $\forall m \in \N$, $\exists y_m \in K_m$ tel que $v(x - y_m) \geq A - \frac{1}{p^m(p-1)}$.
\end{enumerate}
En particulier, si $x$ vérifie $(i)$, il existe $y \in K$ tel que $v(x - y) \geq A-  \frac{1}{p-1}$.
\end{thm}
\begin{proof}
$(i) \Rightarrow (ii)$ On commence par supposer $x\in\bar{K}$. Pour tout $\varepsilon > 0$, on fixe $z_\varepsilon \in K_\infty$ tel que tel que $v(x - z_\varepsilon) \geq A - \varepsilon$, comme dans le corollaire \ref{coro}. On a $ v(\sigma z_\varepsilon - z_\varepsilon) \geq A - \varepsilon$.\\
Soit $m\in\N$. D'après le théorème \ref{ax0}, il existe $y_\varepsilon \in K_m$ tel que $ v(z_\varepsilon-y_\varepsilon) \geq A_m - \varepsilon$. Fixons un tel $y_\varepsilon$, on a alors $ v(x-y_\varepsilon) \geq A_m - \varepsilon$.\\
Si $0<\varepsilon'<\varepsilon$,
$$ v(y_\varepsilon - y_{\varepsilon'}) \geq A_m - \varepsilon.$$
Mais $ep^mv\left(y_{\varepsilon} - y_{\varepsilon'}\right) \in \Z$. Ainsi, pour $\varepsilon$ suffisamment petit (de manière à ce que les entiers immédiatement supérieurs à $ep^m(A_m - \varepsilon)$ et à $ep^mA_m$ soient égaux), on a $v(y_\varepsilon - y_{\varepsilon'}) \geq A_m$.
Fixons un tel $\varepsilon$ et posons $y_m = y_\varepsilon$. Alors, pour tout $0<\varepsilon' < \varepsilon$,
$$ v(x - y_m) \geq \min (v(x-y_{\varepsilon'}),v(y_{\varepsilon'}-y_m)) \geq A_m - \varepsilon'.$$
Cette minoration étant valable pour tout $\varepsilon'$, on a $ v(x - y_m) \geq A_m$.
Maintenant, lorsque $x \in C$, soit $y \in \bar{K}$ tel que $v(x-y)\geq A$. On a alors pour tout $\sigma \in G$, $ v(\sigma y - y) \geq A$. Par conséquent, il existe pour tout $m$ un $y_m \in K_m$ tel que $v(y-y_m) \geq A_m$, et on a pour tout $m \in \N$,
$$v(x - y_m) \geq A - \frac{1}{p^m(p-1)}.$$

$(ii)\Rightarrow(i)$  Supposons d'abord $x\in \bar{K}$. Soit $n \in \N$, pour $m < n$ on pose $z_m = y_m$, et pour $m \geq n$, on pose $z_m = y_n$. On a alors, pour tout $m \in \N$, $ z_m \in K_m$ et $v(y_n - z_m) \geq A_m$. D'après le théorème \ref{ax0}, pour tout $\sigma \in G$, on a $ v(\sigma y_n - y_n) \geq A$. On en déduit immédiatement que pour tout $\sigma \in G$, 
$$ v(\sigma x - x) \geq \min(v(\sigma y_n - y_n),v(\sigma (x-y_n)),v (x-y_n)) \geq A_n.$$
Cette inégalité étant vraie pour tout $n \in \N$, il en résulte que $v(\sigma x - x) \geq A$ quel que soit $\sigma \in G$.\\
On en déduit le résultat pour $x \in C$ comme précédemment.
\end{proof}
Le théorème \ref{ax1} peut se reformuler de la manière suivante :
\begin{cor} Soit $x \in C$. Alors :
$$\inf_{\sigma \in G} v(\sigma x - x) = \sup_{n \in \N} \inf_{y \in K_n} \left\{v(x-y) + \frac{1}{p^n(p-1)}\right\}. $$
\end{cor}
\section{Application au calcul de $H^1(G,\Or_{\bar{K}})$.}
On a la suite exacte $ 0 \rightarrow \Or_{\bar{K}} \rightarrow \bar{K} \rightarrow \bar{K}/\Or_{\bar{K}}\rightarrow 0.$
En passant aux points fixes par $G = \gal(\bar{K}/K)$, on a :
$$ 0  \rightarrow K/\Or_K \rightarrow \left(\bar{K}/\Or_{\bar{K}}\right)^G \rightarrow H^1(G,\Or_{\bar{K}}) \rightarrow 0$$
car $H^1(G, \bar{K}) = 0$ (ici, $H^1(G,\Or_{\bar{K}})$ est muni de la topologie discrète). Il en résulte que $H^1(G,\Or_{\bar{K}})$ est isomorphe au quotient $\left(\bar{K}/\Or_{\bar{K}}\right)^G/\left(K/\Or_K \right)$, identification que l'on fera par la suite.
On a déjà le résultat suivant :
\begin{prop} \label{tue}
Soit $n$ un entier $\geq \frac{e}{p-1}$. Alors $H^1(G, \Or_{\bar{K}})$ est tué par $\pi^n$. En particulier, $H^1(G, \Or_{\bar{K}})$ est tué par $p$.
\end{prop}
\begin{proof} Soit $x \in \left(\bar{K}/\Or_{\bar{K}}\right)^G$. C'est l'image modulo $\Or_{\bar{K}}$ d'un élément $\xi$ de $\bar{K}$ vérifiant pour tout $\sigma \in G$, $v(\sigma \xi - \xi) \geq 0$. Il existe $y \in K$ tel que $v(\xi - y) \geq - \frac{1}{p-1}$. On a alors $v(\pi^n \xi - \pi^n y) \geq \frac{n}{e} - \frac{1}{p-1} \geq 0$, donc $\pi^n \xi = 0$ dans $H^1(G,\Or_{\bar{K}})$, c'est à dire $\pi^nx = 0$.
\end{proof}
\begin{rmq} Ce résultat était déjà connu de Sen, voir \cite{sen}, théorème 3. Il n'implique pas le théorème \ref{ax1}, ni même l'obtention de la constante optimale dans le théorème d'\textsc{Ax}, car $n$ est supposé être entier. Cependant, bien que \textsc{Sen} n'en dise rien, il semble possible d'adapter sa preuve du théorème 3 de \cite{sen} pour en déduire la constante optimale dans le théorème d'\textsc{Ax}, en montrant d'abord que si $x \in \bar{K}$ et $\sigma x - x \in \Or_{\bar{K}}$ pour tout $\sigma \in G$, alors il existe $y \in K$ tel que $v(x - y) \geq -\frac{1}{p-1}$.\end{rmq}

Dans la suite, on note
$$\Ar_n = \left\{z \in \bar{K} ~/~ v(z) \geq - \frac{1}{p^n(p-1)}\right\}.$$
\subsection{Cas non ramifié.}
Dans cette sous-partie, on suppose que $K=F$, c'est à dire $K/F$ absolument non ramifiée.\\
On rappelle que l'on identifie $\left(\bar{K}/\Or_{\bar{K}}\right)^G/\left(K/\Or_K \right)$ et $H^1(G,\Or_{\bar{K}})$. On va montrer que l'on peut associer à $x \in H^1(G,\Or_{\bar{K}})$ une suite d'éléments de $k$ dont nous étudierons ensuite les propriétés.  Pour $n \in \N$, on note $\eta_n = \pi_n^{-1}$.
\begin{prop}
Soit $x \in H^1(G,\Or_{\bar{K}})$, il existe un antécédent $\xi$ de $x$ dans $\bar{K}$ et une suite $(x_n)_{n \in \N^*}$ d'éléments de $k$ telle que pour tout $n \in \N$,
$$ \xi = \sum_{i=1}^n [x_i]\eta_i \mod \Ar_n,$$
où pour tout $i$, $[x_i]$ est le représentant de Teichmüller de $x_i$. De plus, la suite $(x_n)$ ne dépend que de $x$. 
\end{prop}
\begin{proof}
Soit $\xi$ un antécédent de $x$ dans $\bar{K}$. D'après le théorème \ref{ax1}, il existe pour tout $m \in \N$ un $y_m \in K_m$ tel que $\xi = y_m \mod \Ar_m$. On fixe une telle famille $(y_m)$. Comme $v(\xi - y_0) \geq -\frac{1}{p-1}$, on peut supposer, quitte à remplacer $\xi$ par $\xi - y_0$, que $v(\xi) \geq  -\frac{1}{p-1}$. On construit la suite $(x_n)_{n \in \N}$ par récurrence. Le cas $n=0$ est trivial.\\
On suppose construite la suite jusqu'à l'indice $n$, et on écrit :
$$ \xi = [x_1]\eta_1 +\cdots + [x_n]\eta_n + z, \textrm{ avec } z \in \Ar_n, \textrm{ et } y_{n+1} = \sum_{i \geq -n_0}c_i'\pi_{n+1}^i,$$
où les $c_i'$ sont pris parmi les représentants de Teichmüller des éléments de $k$ et $n_0 \geq 0$. En posant $c_i = c_{-i}'$,  on a $ y_{n+1} = \sum_{i=1}^{n_0} c_i\eta_{n+1}^i \mod \Or_{\bar{K}}$. Il en résulte que $ \xi = \sum_{i=1}^{n_0} c_i\eta_{n+1}^i + z'$,  avec $z' \in \Ar_{n+1}$. Pour $k \geq 1$,
$$ \eta_{n+1}^k \in \Ar_n \Leftrightarrow (p\geq3 \textrm{ et } k=1) \textrm{ ou } (p=2 \textrm{ et } k \in \{1,2\}).$$
Ainsi, pour $p \geq 3$, en réduisant modulo $\Ar_n$, on a  $ \sum_{i=2}^{n_0} c_i\eta_{n+1}^i = [x_1]\eta_1 +\cdots+[x_n]\eta_n \mod \Ar_n$. Par conséquent, en identifiant les coefficients dans $K_{n+1}$, on a
$$x = [x_1]\eta_1 +\cdots+[x_n]\eta_n + c_1\eta_{n+1} + z', $$
ce qui achève la récurrence en posant $x_{n+1} = c_1 \mod p$.\\
Lorsque $p = 2$, on a (toujours en réduisant modulo $\Ar_n$) $ \sum_{i=3}^{n_0} c_i\eta_{n+1}^i = [x_1]\eta_1 +\cdots+[x_{n-1}]\eta_{n-1}$. Ainsi, $\xi = c_1\eta_{n+1} + c_2\eta_{n+1}^2 + [x_1]\eta_1 +\cdots+[x_{n-1}]\eta_{n-1} + z'$. Mais $\eta_{n+1}^2 = \eta_n$, et donc $\xi$ s'écrit encore
$ \xi = [x_1]\eta_1 +\cdots+[x_n]\eta_n \mod \Ar_{n+1}$ (car $\eta_{n+1} \in \Ar_{n+1}$).\\
La suite $(x_n)_{n\in\N}$ ainsi associée à $x \in H^1(G,\Or_{\bar{K}})$ ne dépend pas de $\xi \in \Ar_0$. De plus si $ \xi \mod \Ar_n= \sum_{i=1}^n [x_i]\eta_i = \sum_{i=1}^n [x_i']\eta_i$,
alors en réduisant modulo $\Ar_i$, on a
$$ \sum_{j=1}^i [x_j]\eta_j = \sum_{j=1}^i [x_j']\eta_j \textrm{ si } p\geq 3, \textrm{ et }\sum_{j=1}^{i-1} [x_j]\eta_j = \sum_{j=1}^{i-1}[x_j']\eta_j \textrm{ si } p=2.$$
On en déduit par récurrence sur $i$ que pour tout $i \in \N$,  $v([x_i] - [x_i']) \geq \frac{1}{p^i} > 0$. Comme les $[x_i]$ sont dans $K$, ils sont égaux modulo $p$, et la suite $(x_n)$ est unique.
\end{proof}
 On peut donc définir l'application
$$\begin{array}{rrcl}
\psi ~:~ &H^1(G,\Or_{\bar{K}}) &\longrightarrow &k^{\N^*} \\
&x &\mapsto &(x_n)_{n \in \N^*}
\end{array}$$
telle que pour tout $n \in \N^*$, $x - \sum_{i=1}^n [x_i]\eta_i \in \Ar_n$.
C'est un morphisme $\Or_K$-linéaire (ou $k$-linéaire puisque $H^1(G,\Or_{\bar{K}})$ est tué par $p$), injectif. Il nous reste à identifier son image. On va adapter à notre cas des constructions proposées par Kedlaya dans un cadre une peu différent (il s'agissait de donner une description d'une clôture algébrique de $\bar{\F_p}((t))$, voir \cite{ked1}). Compte tenu du fait que nous étudions des objets plus simples, nous avons préféré réécrire l'étude de Kedlaya dans le langage de notre problème.
\begin{defi}
Soit $(x_n)_{n \in \N^*} \in k^{\N^*}$. On dit que la suite $(x_n)$ est \emph{twist-récurrente} s'il existe $d_0,\dots,d_r \in k$ non tous nuls tels que
$$ \forall n \in \N^*,~ d_0x_n + d_1x_{n+1}^p + \cdots + d_r x_{n+r}^{p^r} = 0.$$
\end{defi}
\begin{prop} L'application $\psi$ définie précédemment induit un isomorphisme de $H^1(G,\Or_{\bar{K}})$ sur le sous-espace de $k^{\N^*}$ formé des suites twist-récurrentes. 
\end{prop}
\begin{proof} On commence par montrer que si $x \in H^1(G,\Or_{\bar{K}})$, alors $\psi(x)$ est twist-récurrente. L'élément $x$ provient d'un élément $\xi_0 \in \bar{K}$, que l'on peut supposer de valuation $\geq-\frac{1}{p-1}$.
Calculons $\xi_0^p$ lorsque $\psi(x) = (x_n)$. Soit $n \in \N^*$, écrivons $ \xi_0 = \sum_{i=1}^{n} [x_i]\eta_i + z$, avec $z \in \Ar_n$ et les $[x_i]$ les représentants de Teichmüller. Alors $ \xi_0^p = \sum_{i=0}^{n-1}[x_{i+1}]^p\eta_i + \tilde{z}$, avec $\tilde{z} = \sum_{j=1}^{p}\binom{p}{j} \left(\sum_{i=1}^n [x_i]\eta_i\right)^{p-j}z^j$. Or pour tout $j \in \{1,\dots,p-1\}$, 
$$v\left(\binom{p}{j} \left(\sum_{i=1}^n [x_i]\eta_i\right)^{p-j}z^j\right) \geq 1 - \frac{j}{p^n(p-1)} - \frac{p-j}{p^n}\geq 0.$$
Ainsi, $ \xi_0^p = [x_1]^p\eta_0 + \sum_{i=1}^{n-1}[x_{i+1}]^p\eta_i \mod \Ar_{n-1}$. En particulier, $\xi_0^p$ vérifie $v(\sigma \xi_0^p - \xi_0^p) \geq 0$ pour tout $\sigma \in G$, et se réduit dans $H^1(G,\Or_{\bar{K}})$ sur un élément dont l'image par $\psi$ est $(x_{n+1}^p)$. Pour tout $s \in \N$, on pose $\xi_{s+1} = \xi_s^p - [x_s]^{p^s}\eta_0$. On vérifie facilement que $v(\sigma \xi_s^p - \xi_s^p) \geq 0$ pour tout $\sigma \in G$, et que la réduction dans $H^1(G,\Or_{\bar{K}})$ a pour image par $\psi$ la suite $[x_{n+s}]^{p^s}$. Comme par ailleurs  tous les $\xi_s$ sont dans $K(\xi_0)$, ils forment une famille liée sur $K$. Il existe donc $\delta_0,\dots,\delta_r \in K$ tels que $\delta_0\xi_0 + \cdots + \delta_r \xi_r = 0.$ Quitte à multiplier par une puissance de $p$, on peut supposer que les $\delta_s$ sont dans $\Or_K$ et qu'au moins l'un d'entre eux est non divisible par $p$. L'application $\psi$ étant $\Or_K$-linéaire, on en déduit que pour tout $n\in \N^*$,
$$d_0x_n + d_1x_{n+1}^p + \cdots + d_r x_{n+r}^{p^r} = 0,$$
où $d_s$ désigne la réduction de $\delta_s$ modulo $p$. 
Donc $(x_n)$ est twist-récurrente.

Réciproquement, il nous reste à prouver que si $(x_n)_{n \in \N^*}$ est twist-récurrente, alors c'est l'image par $\psi$ d'un élément de $H^1(G,\Or_{\bar{K}})$. Soit donc $(x_n)_{n \in N^*}$ twist-récurrente et soient $d_0,\dots,d_r \in k$ non tous nuls tels que
$$ \forall n \in \N^*,~ d_0x_n + \cdots + d_rx_{n+r}^{p^r} = 0.$$
On note $\delta_0,\dots,\delta_r$ les représentants de Teichmüller des $d_k$ dans $\Or_K$, et pour $n \geq 1$, on note $[x_n]$ le représentant de Teichmüller de $x_n$ dans $\Or_K$,  et on considère le polynôme 
$$P = -\left(\delta_r[x_r]^{p^r}\eta_1 + \cdots +(\delta_1[x_r]^p+ \cdots + \delta_r[x_{2r-1}]^{p^r})\eta_r\right)+ \delta_0X +\cdots + \delta_rX^{p^r}.$$
On va montrer qu'il admet une racine dans $\Or_{\bar{K}}$ dont l'image dans $H^1(G,\Or_{\bar{K}})$ s'envoie par $\psi$ sur la suite $(0,\dots, x_{r+1},x_{r+2},\dots)$ (la suite commence par $r$ zéros). On fixe $n\geq 1$ et on cherche une racine $\xi$ de $P$ sous la forme $ \xi= [x_{r+1}]\eta_{r+1} + \cdots + [x_{n+r}]\eta_{n+r} + y\eta_{n+r+1}$ avec $y \in \Or_{\bar{K}}$.\\
Tout d'abord, on a pour $0 \leq s \leq r$,
$$\xi^{p^s} = \left(\sum_{i=r+1}^{n+r}[x_i]\eta_i\right)^{p^s} + \eta_{n+r+1}^{p^s}y^{p^s} + \sum_{j=1}^{p^s-1} \binom{p^s}{j}\left(\sum_{i=r+1}^{n+r}[x_i]\eta_i\right)^{p^s - j}\eta_{n+r+1}^jy^j.$$
Comme $P(\xi) =0$, on a :
\begin{multline*} \sum_{k=0}^r \delta_k\left[  \left(\sum_{i=r+1}^{n+r}[x_i]\eta_i\right)^{p^k} + \sum_{j=1}^{p^k-1} \binom{p^k}{j}\left(\sum_{i=r+1}^{n+r}[x_i]\eta_i\right)^{p^k - j}\eta_{n+r+1}^jy^j + \eta_{n+r+1}^{p^k}y^{p^k}\right] \\= \delta_r[x_r]^{p^r}\eta_1 + \cdots +(\delta_1[x_r]^p+ \cdots + \delta_r[x_{2r-1}]^{p^r})\eta_r.
\end{multline*}
On voit donc que $y$ est annulé par un polynôme $Q$ de degré $p^r$, dont le coefficient constant est 
$$\sum_{k=0}^r \delta_k\left(\sum_{i=r+1}^{n+r}[x_i]\eta_i\right)^{p^k} - (\delta_r[x_r]^{p^r}\eta_1 + \cdots +(\delta_1[x_r]^p+ \cdots + \delta_r[x_{2r-1}]^{p^r})\eta_r),$$
et dont le coefficient dominant est $\delta_r\eta_{n+r+1}^{p^r} = \delta_r\eta_{n+1}$. Par ailleurs, on a :
$$ \left(\sum_{i=r+1}^{n+r}[x_i]\eta_i\right)^p = \sum_{\varepsilon_1 +\cdots + \varepsilon_n = p} \frac{p!}{\varepsilon_1!\dots\varepsilon_n!} [x_{r+1}]^{\varepsilon_1}\dots[x_{n+r}]^{\varepsilon_n} \eta_{r+1}^{\varepsilon_1}\dots\eta_{n+r}^{\varepsilon_n}.$$
On remarque que si tous les $\varepsilon_i$ sont $<p$, $v(\frac{p!}{\varepsilon_1!\dots\varepsilon_n!} ) = 1$ et $v( \eta_{r+1}^{\varepsilon_1}\dots\eta_{n+r}^{\varepsilon_n}) = -\frac{\varepsilon_1}{p^{r+1}} - \cdots - \frac{\varepsilon_n}{p^{n+r}} > -1$. Donc tous les termes correspondants de la somme sont nuls modulo $\Or_{\bar{K}}$, et donc
$$  \left(\sum_{i=r+1}^{n+r}[x_i]\eta_i\right)^p =[x_{r+1}]^p\eta_r +  [x_{r+2}]^p\eta_{r+1} + \cdots + [x_{n+r}]^p\eta_{n+r-1}\mod \Or_{\bar{K}}.$$
Un calcul analogue montre que pour tout $s \in \{1,\dots,r\}$, on a :
$$ \left(\sum_{i=r+1}^{n+r}[x_i]\eta_i\right)^{p^s} = [x_{r+1}]^p\eta_{r+1-s} +  [x_{r+2}]^p\eta_{r+2-s} + \cdots + [x_{n+r}]^p\eta_{r+n-s}\mod \Or_{\bar{K}} .$$
En additionnant, on a donc :
$$ \sum_{s=0}^r \delta_s\left(\sum_{i=r+1}^{n+r}[x_i]\eta_i\right)^{p^s} = \delta_r[x_r]^{p^r}\eta_1 + \cdots +(\delta_1[x_r]^p+ \cdots + \delta_r[x_{2r-1}]^{p^r})\eta_r +\eta_{n+1}z,$$
avec $z \in \Or_{\bar{K}}$. Le coefficient constant de $Q$ est donc de valuation $\geq -\frac{1}{p^{n+1}}$ qui est la valuation de son coefficient dominant. En traçant son polygone de Newton, on en déduit que ce polynôme a une racine de valuation positive. Le polynôme $P$ a donc bien une racine de la forme $x =  [x_{r+1}]\eta_{r+1} + \cdots + [x_{n+r}]\eta_{n+r} + y\eta_{n+r+1}$ avec $y \in \Or_{\bar{K}}$. Comme $P$ a $p^r$ racines, il y a au moins l'une d'entre elles qui est obtenue pour une infinité de $n$ à l'aide de la construction précédente. Notons $\xi_0$ une telle racine. En réduisant modulo $\Ar_n$, on voit que pour \emph{tout} $n \in \N^*$, on a $ \xi_0 - \sum_{i=r+1}^{r+n} [x_i]\eta_i \in \Ar_n$. On a donc $\xi_0 \in H^1(G,\Or_{\bar{K}})$, et $\psi(\xi_0) = (0,0,\dots,0,x_{r+1},x_{r+2},\dots,x_n,\dots)$. Comme $(x_1,\dots,x_r,0,0,\dots)$ est l'image de $\sum_{i=1}^r [x_i]\eta_i \in H^1(G,\Or_{\bar{K}})$, on en déduit que $(x_n)$ est dans l'image de $\psi$.
\end{proof}
\begin{cor}\label{twist}
Si $K=F$, $H^1(G,\Or_{\bar{K}})$ est un $k$-espace vectoriel de dimension infinie. Plus précisément, si $k$ est fini, la dimension de $H^1(G,\Or_{\bar{K}})$ est dénombrable. Si $k$ est infini, elle est égale à la cardinalité de $k$.
\end{cor}
\begin{proof}  $H^1(G,\Or_{\bar{K}})$ est l'ensemble des suites twist-récurrentes à valeurs dans $k$. Pour $d_0,\dots,d_r$ non tous nuls fixés, l'ensemble des suites vérifiant la relation de twist-récurrence
$$ \forall n \in \N, ~ d_0x_n + \cdots + d_rx_{n+r}^{p^r} = 0$$
forme un $k$-espace vectoriel de dimension finie. La réunion des espaces déterminés par l'ensemble des $(d_0,\dots,d_r) \in k^{r+1}$, pour $r \geq 0$, est de dimension dénombrable si $k$ est fini, et de dimension la cardinalité de $k$ si $k$ est infini.
\end{proof}
\subsection{Vers le cas ramifié.}
Dans cette partie, on entame une étude analogue à la précédente, en ne supposant plus cette fois-ci que $e = 1$. Les preuves sont souvent esquissées.

\subsubsection{Cas $e\leq p-1$.}
On étudie ici le cas où $e\leq p-1$, pour lequel les constructions proposées dans le cas non ramifié s'adaptent facilement. La proposition précédente nous dit que $ H^1(G,\Or_{\bar{K}})$ est tué par $\pi$, en particulier il a une structure naturelle de $k$-espace vectoriel. On peut en fait associer à un élément de $H^1(G,\Or_{\bar{K}})$ une famille de suites d'éléments de $k$.
\begin{prop}
Soit $x \in H^1(G,\Or_{\bar{K}})$, il existe $e$ suites $(x_{1,n}),\dots,(x_{e,n})$ d'éléments de $k$ telles que pour tout $n \in \N$, $x = \sum_{i=1}^n\sum_{j=1}^e [x_{j,i}]\eta_i^j \mod \Ar_n$, $[x_{j,n}]$ désignant le représentant de Teichmüller de $x_{j,n}$.
\end{prop}
\begin{proof} On procède comme dans le cas non ramifié, la différence ici étant que la réduction modulo $\Ar_n$ de $y_{n+1}$ ne tue plus seulement $c_1\eta_{n+1}$ mais la somme $c_1\eta_{n+1} + \cdots + c_e\eta_{n+1}^\rho$ lorsque $e< p-1$, $\rho$ désignant le plus grand entier inférieur à $\frac{ep}{p-1}$. Le cas $e=p-1$ entraîne la même modification que dans le cas $e = 1, p=2$.
\end{proof}
On a donc comme précédemment une application :
$$\begin{array}{rrcl}
\psi ~:~ &H^1(G,\Or_{\bar{K}}) &\longrightarrow &\left(k^{\N^*}\right)^e \\
&x &\mapsto &((x_{1,n})_{n \in \N^*},\dots,(x_{e,n})_{n \in \N^*})
\end{array}$$
\begin{thm}
L'application qui à $x \in H^1(G,\Or_{\bar{K}})$ associe $((x_{1,n}),(x_{2,n}),\dots, (x_{e,n}))$ est injective, et son image est l'ensemble des $e$-uplets de suites twist-récurrentes à valeurs dans $k$, qui est donc isomorphe à $H^1(G, \Or_{\bar{K}})$.
\end{thm}
\begin{proof}Nous ne donnerons pas ici la démonstration de ce résultat, les idées sont similaires à celles de la preuve dans le cas non ramifié.
\end{proof}
\subsubsection{Cas général.} 
Il nous reste à traiter le cas $e \geq p$. $H^1(G, \Or_{\bar{K}})$ n'est pas tué par $\pi$. On note $\tau = \left\lfloor \frac{e}{p-1}\right\rfloor$ et $\rho = \left\lfloor \frac{ep}{p-1}\right\rfloor$ (où $\left\lfloor t \right\rfloor$ désigne le plus grand entier inférieur ou égal à $t$). Un calcul direct montre que $\rho - \tau = e$. Soit $x \in H^1(G, \Or_{\bar{K}})$. On dispose d'une suite $(y_n)$ avec pour tout $n \in \N$, $y_n \in K_n$ et $x = y_n \mod \Ar_n$. On écrit pour tout $n \in \N$ que $y_n = \sum_{j=1}^{N_n} c_{j,n} \eta_n^j$, avec les $c_{j,n}$ pris dans une famille de représentants des éléments de $k$ dans $\Or_K$ de valuation nulle ou infinie (on ne suppose plus ici qu'il s'agit des représentants de Teichmüller). Pour tout $n \in \N^*$, il existe $z_n\in \Ar_n$ tel que $ x = \sum_{j=1}^{N_n} c_{j,n}\eta_n^j + z_n$. On remarque que $\eta_n^j \in \Ar_n $ si et seulement si $j \leq \tau$, on peut donc supposer que la somme commence à $j=\tau+1$.\\
Réduisons modulo $\Ar_n$ l'égalité précédente écrite aux rangs $n$ et $n+1$. On a pour $j \geq 1$ :
$$ \eta_{n+1}^j \in \Ar_n \Leftrightarrow \frac{j}{ep^{n+1}} \leq \frac{1}{p^n(p-1)} \Leftrightarrow j \leq \frac{ep}{p-1} \Leftrightarrow j \leq \rho. $$
Par conséquent,
$$\sum_{j=\tau+1}^{N_n}c_{j,n} \eta_n^j = \sum_{j=\rho+1}^{N_{n+1}}c_{j,n+1} \eta_{n+1}^j \mod \Ar_n$$
Il découle de ces calculs la proposition suivante :
\begin{prop}
Soit $x \in  H^1(G,\Or_{\bar{K}})$. Il existe $e$ suites $(\alpha_{n, \tau+1})_{n \in \N^*},\dots,(\alpha_{n,\rho})_{n \in \N^*}$ d'éléments de $\Or_K$ telles que pour tout $n \in \N^*$, $x = \sum_{i=1}^n \sum_{j = \tau +1}^\rho \alpha_{i,j}\eta_i^j \mod \Ar_n$.
\end{prop}
\begin{proof} On procède par récurrence, en conservant les notations précédentes. Le cas $n=1$ est immédiat, compte tenu du fait que pour $j \geq \rho$, $\eta_1^j \in K$.\\
Pour $n \geq 1$, on a $ \sum_{j= \rho +1}^{N_{n+1}} c_{j,n+1} \eta_{n+1}^i = \sum_{i=1}^n \sum_{j = \tau +1}^\rho \alpha_{i,j}\eta_i^j \mod \Ar_n$ par hypothèse de récurrence, et donc
$$ x = \sum_{k=\tau + 1}^\rho c_{k, n+1} \eta_i^k+\sum_{i=1}^n \sum_{k = \tau +1}^\rho \alpha_{i,k}\eta_i^k+ z_{n+1},$$
ce qui prouve la proposition en posant $\alpha_{n+1,j} = c_{j,n+1}$.
\end{proof}
Remarquons que dans cette écriture, un $\eta_i^j$ n'apparaît qu'une fois ; autrement dit, il n'existe pas de couples $(i,j)$, $(i',j')$ distincts d'indices dans cette somme tels que $\eta_i^j = \eta_{i'}^{j'}$. En effet, on a :
$$ p(\tau + 1) > p\frac{e}{p-1} \geq \rho.$$
En conséquence, si $j,j' \in \{\tau+1, \dots, \rho\}$ et $j<j'$, on a $pj>j'$, et donc $p^{i'}j \neq p^i j'$ pour tous $i,i'$. De plus, comme on calcule modulo $K$, on peut supprimer de cette somme les $\alpha_{j,i}$ tels que $j|p^i$ ; on peut également regrouper les termes indexés par $(i,j-\lambda p^i)$ pour $\lambda \in \N$ car alors $\eta_i^{j-\lambda p^i} = \pi^\lambda \eta_i^j$. 	Pour $i \in \N^*$ et $j \in \{ \tau + 1, \dots, \rho\}, j\nmid p^i$, on note $\gamma_{i,j} = \max\{ s \in \{ \tau + 1, \dots, \rho\} ~/~p^i \mid s- j \}$ et $I = \{(i,\gamma_{i,j}) ~/~ i \in \N^*, j \in \{\tau+1,\dots,\rho\},p^i \nmid \gamma\}$. On peut encore écrire pour tout $n \in \N$,
$$ x = \sum_{(i,\gamma)\in I, i \leq n} \beta_{i,\gamma}\eta_i^\gamma \mod \Ar_n,$$
en regroupant les termes comme expliqué précédemment. Une telle écriture est alors unique : il n'y a aucune sous-somme finie de $\sum_{(i,\gamma)\in I} \beta_{i,\gamma}\eta_i^\gamma$ égale à un élément non nul de $K$, et aucun $\beta_{i,\gamma}\eta_i^\gamma$ non nul n'est de valuation positive. En effet, si $\gamma - \lambda p^i \in \{\tau+1,\dots,\rho\}$ et $\gamma - (\lambda+1)p^i \notin \{\tau+1,\dots,\rho\}$, alors dès que $\beta_{i,\gamma}$ est non nul, on a $e v(\beta_{i,\gamma}) \leq \lambda$, et donc $e v(\beta_{i,\gamma}) < \lambda + \frac{\tau}{p^i} \leq - ev(\eta_i^\gamma)$.
On peut donc bien définir une application de $H^1(G,\Or_{\bar{K}})$ vers $\Or_K^I$, qui à $x$ associe les $\beta_{i, \gamma}$.

\begin{prop} Soit $r\geq 1$ et $H_{\pi^r}$ le sous-module de $\pi^r$-torsion de $H^1(G, \Or_{\bar{K}})$, alors $H^1(G, \Or_{\bar{K}})/H_{\pi^r}$ est un $\Or_K$-module de type fini engendré par au plus $\frac{pe}{r(p-1)^2}$ éléments.
\end{prop}
\begin{proof} On note $I_r = \{(i,\gamma)\in I, ~ rp^i < \gamma \}$. La $\pi^r$-torsion de $H^1(G, \Or_{\bar{K}})$ est l'ensemble des $x \in H^1(G, \Or_{\bar{K}})$ tels que $\pi^r x = 0$, c'est à dire $v(\pi^rx) \geq 0$. En associant à $x$ la famille des $\beta_{i,\gamma}$ et en considérant les valuations, $\pi^rx = 0$ si et seulement si pour tout $(i,\gamma) \in I$, $\frac{r}{e} + v(\beta_{i,\gamma}) \geq \frac{\gamma}{ep^i}$.
Notons $H_{\pi^r}$ le sous-module de $\pi^r$-torsion de $H^1(G, \Or_{\bar{K}})$, la proposition et les calculs précédents montrent que l'application composée
$$ \Or_K^{I_r} \rightarrow \sum_{(i,\gamma) \in I_r} \eta_i^\gamma \Or_K \rightarrow H^1(G, \Or_{\bar{K}})/H_{\pi^r} $$
est surjective. Le $\Or_K$-module $H^1(G, \Or_{\bar{K}})/H_{\pi^r}$ est donc de type fini, engendré par des éléments en nombre fini majoré par le cardinal de $I_r$. \`A $i$ fixé, il y a au plus $p^i$ éléments de la forme $(i,\gamma)$ dans $I$.  De plus, si $(i,\gamma) \in I_r$,  $p^i < \frac{\rho}{r}$ et donc $i < \log_p(\rho/r)$. Le cardinal de $I_r$ est donc majoré par $\sum_{1\leq i < \log_p(\rho/r)} p^i$. Cette somme est majorée par $\frac{\log_p(\rho/r) }{p-1} \leq \frac{pe}{r(p-1)^2}$.
\end{proof}
\begin{rmq} Le $\Or_K$-module $H_\pi$ a une structure naturelle de $k$-espace vectoriel, on a un morphisme injectif $H_\pi \rightarrow k^\N$, et il semble possible d'espérer que les résultats de Kedlaya s'adaptent pour montrer que l'image de cette injection peut se décrire en terme de suites twist-récurrentes dans leur définition la plus générale (voir \cite{ked1} et \cite{ked2}). Nous ne ferons pas ici cette interprétation.\end{rmq}
\begin{rmq} Les méthodes utilisées dans cet article peuvent peut-être se généraliser au calcul de $H^1(G, GL_n(\Or_{\bar{K}}))$. En effet, en utilisant la suite exacte
$$1 \rightarrow GL_n(K)/GL_n(\Or_K) \rightarrow GL_n(\bar{K})/GL_n(\Or_{\bar{K}}) \rightarrow H^1(G,GL_n(\Or_{\bar{K}})) \rightarrow 1,$$ 
et la réduction de Hermite des éléments de $GL_n(\Or_{\bar{K}})$, on devrait pouvoir donner une description de $H^1(G, GL_n(\Or_{\bar{K}}))$, ce qui pourrait donner un avatar de la théorie de Sen dans le cas des représentations de torsion. \end{rmq}
\section*{Remerciements.}
Je tiens à remercier ici \textsc{Xavier Caruso} \footnote{IRMAR, Université de Rennes 1, Campus de Beaulieu, 35042 Rennes Cedex, France} pour ses nombreuses et pertinentes remarques, ses relectures constructives de ma prose, et en particulier pour la suggestion qui est la clé de ce travail, à savoir le fait de s'intéresser aux extensions engendrées par des racines de l'uniformisante. Je remercie également le referee pour ses remarques, ainsi que sa relecture minutieuse de ce travail.

\bibliographystyle{alpha}
\bibliography{bib_optimast}
\end{document}